\documentclass[11pt,leqno]{article}
\usepackage{amssymb,amsfonts}
\usepackage{amsmath,amsthm,amsxtra}
\usepackage[all]{xy}
\usepackage{color}
\usepackage{verbatim}

\usepackage{enumerate}

\newcommand\bigcheck[1]{#1 \raise1ex\hbox{$\hspace{-1ex}{}^\vee$}}
\newcommand\sucheck[1]{#1 \raise0.5ex\hbox{$\hspace{-1ex}{}^\vee$}}

\newtheorem{theorem}{Theorem}[section]
\newtheorem{lemma}[theorem]{Lemma}
\newtheorem{corollary}[theorem]{Corollary}
\newtheorem{proposition}[theorem]{Proposition}
\newtheorem*{lemma*}{Lemma}

\theoremstyle{definition}
\newtheorem{definition}[theorem]{Definition}

\theoremstyle{remark}
\newtheorem{remark}[theorem]{Remark}
\newtheorem{example}[theorem]{Example}

\newcommand{\mc}[1]{{\mathcal #1}}

\newcommand{\mb}[1]{{\mathbb #1}}

\newcommand{\id}{{1 \mskip -5mu {\rm I}}}

\renewcommand{\tilde}{\widetilde}

\renewcommand{\ker}{\mathop{\rm Ker }}

\definecolor{light}{gray}{.9}

\begin{document}


\title{Rational matrix pseudodifferential operators}

\author{
Sylvain Carpentier
\thanks{Ecole Normale Superieure, 
75005 Paris, France, and M.I.T~~
sylvain.carpentier@ens.fr ~~~~},~~
Alberto De Sole
\thanks{Dipartimento di Matematica, Universit\`a di Roma ``La Sapienza'',
00185 Roma, Italy ~~
desole@mat.uniroma1.it ~~~~
Supported in part by Department of Mathematics, M.I.T.},~~
Victor G. Kac
\thanks{Department of Mathematics, M.I.T.,
Cambridge, MA 02139, USA.~~
kac@math.mit.edu~~~~
Supported in part by an NSF grant~~ 
}~~
}

\maketitle


\begin{abstract}
\noindent 
The skewfield $\mc K(\partial)$ of rational pseudodifferential operators 
over a 
differential field $\mc K$ is the skewfield of fractions of the algebra
of differential operators $\mc K[\partial]$. 
In our previous paper we showed that any 
$H\in \mc K(\partial)$ has a minimal fractional decomposition $H=AB^{-1}$,
where $A,B\in \mc K[\partial]$, $B\neq 0$, and any common right
divisor of $A$ and $B$ is a non-zero element of $\mc K$. 
Moreover, any right fractional decomposition of $H$ is obtained by multiplying
$A$ and $B$ on the right by the same non-zero element of $\mc K[\partial]$. 
In the present paper we study
the ring $M_n(\mc K(\partial))$ of $n\times n$ matrices over the skewfield
$\mc K(\partial)$. We show that similarly, any  
$H\in M_n(\mc K(\partial))$ has a minimal fractional decomposition $H=AB^{-1}$,
where $A,B\in M_n(\mc K[\partial])$, $B$ is non-degenerate, and 
any common right divisor of $A$ and $B$ is an invertible element of the ring 
$M_n(\mc K[\partial])$. Moreover, 
any right fractional decomposition of $H$ is obtained by multiplying
$A$ and $B$ on the right by the same non-degenerate element of 
$M_n(\mc K [\partial])$. We give several equivalent definitions of the 
minimal fractional
decomposition. These results are applied to the study of maximal isotropicity
property, used in the theory of Dirac structures.  
\end{abstract}

\section{Introduction}
\label{sec:intro}
Let $\mathcal{K}$ be a differential field with derivation $\partial$ and let 
${\mathcal{K}}[\partial]$ be the algebra of differential operators over 
$\mathcal{K}$. The skewfield ${\mathcal{K}}(\partial)$ of rational 
pseudodifferential operators is, by definition,
the subskewfield of the skewfield of pseudodifferential operators 
${\mathcal{K}}(({\partial}^{-1}))$, generated by the subalgebra 
${\mathcal{K}}[\partial]$. In our paper [CDSK12]
we showed that any rational pseudodifferential operator $H$ has a unique right minimal
decomposition $H=AB^{-1}$, where $A,B \in {\mathcal{K}}[\partial]$, $B$ is a non-zero monic
differential operator, and any other right fractional decomposition of $H$ can be obtained 
by multiplying on the right both $A$ and $B$ by a non-zero differential 
operator $D$.

In the present paper we establish a similar result for the ring 
$M_n({\mathcal{K}}(\partial))$ of $n \times n$ matrix rational pseudodifferential operators. Namely we show that 
any $H \in M_n({\mathcal{K}}(\partial))$ has a 
\emph{right minimal fractional decomposition} $H=AB^{-1}$, where 
$B \in M_n({\mathcal{K}}[\partial])$ is non-degenerate (i.e. has a non-zero 
Dieudonn\'e 
determinant $\det (B)$), satisfying one of the following equivalent 
properties :
\begin{enumerate}[$(i)$]
\item
$d(B)$ is minimal among all possible right fractional decompositions 
$H=AB^{-1}$, where $d(B)$ is the order of $\det(B)$ ;
\item
$A$ and $B$ are \emph{coprime}, i.e.
if $A=A_1D$ and $B=B_1D$, with 
$A_1,B_1,D\in M_n(\mc K[\partial])$,
then $D$ is invertible in $M_n(\mc K[\partial])$;
\item
$\ker A\cap\ker B=0$ in any differential field extension of $\mathcal{K}$.
\end{enumerate}
By $(i)$, a right minimal fractional decomposition exists for any $n \times n$ 
matrix 
rational pseudodifferential operator $H$. We prove its uniqueness, namely 
that all right minimal fractional decompositions can be obtained from each 
other by multiplication on the right of the numerator and the denominator 
by an invertible $n \times n$ matrix differential operator $D$.
Moreover, any right fractional decomposition of $H$ can be obtained by 
multiplying on the right the numerator and the denominator of a minimal right 
fractional decomposition by the same non-degenerate matrix differential 
operator. 

We derive from these results the following maximal isotropicity property of 
the 
minimal fractional decomposition $H=AB^{-1}$, which is important for the 
theory of Dirac structures [D93], [BDSK09], [DSK12]. 
Introduce the following bilinear form on the space 
${{\mathcal{K}}^n}\oplus {{\mathcal{K}}^n}$ with values in 
${\mathcal{K}}/{\partial{\mathcal{K}}}$ :
$$ ( P_1 \oplus Q_1 | P_2 \oplus Q_2)= \int ({P_1}.{Q_2} + {P_2}.{Q_1}),$$
where $\int$ stands for the canonical map 
$\mathcal{K} \rightarrow {\mathcal{K}}/{\partial{\mathcal{K}}}$ and 
$P.Q$ is the standard dot product. Let $A$ and $B$ be two $n\times n$
matrix differential operators. Define 
$$ {\mathcal{L}}_{A,B}=\{ B(\partial)P \oplus A(\partial)P | 
P \in {\mathcal{K}}^n \}.$$
It is easy to see that, assuming that $\det (B)\neq 0$, the subspace 
${\mathcal{L}}_{A,B}$ of
${{\mathcal{K}}^n}\oplus {{\mathcal{K}}^n}$ is isotropic if and only if the 
matrix rational pseudodifferential operator $H=AB^{-1}$ is skewadjoint. 
We prove that  ${\mathcal{L}}_{A,B}$ 
is maximal isotropic if $AB^{-1}$ is a right minimal fractional 
decomposition of $H$. Note that ${\mathcal{L}}_{A,B}$ is independent of the 
choice of the minimal fractional decomposition due to its uniqueness, 
mentioned above.

We wish to thank Pavel Etingof and Andrea Maffei for useful discussions,
and Mike Artin, Michael Singer and Toby Stafford for useful correspondence.

\section{Some preliminaries on rational pseudodifferential operators}

Let $\mathcal{K}$ be a differential field of characteristic $0$, with a 
derivation $\partial$, and let $\mathcal{C}=\ker\partial$ be the subfield of constants.
Consider the algebra ${\mathcal{K}}[\partial]$ (over $\mathcal{C}$) of 
differential operators. It is a subalgebra of the skewfield 
${\mathcal{K}}(({\partial}^{-1}))$ of pseudodifferential operators. 
The subskewfield ${\mathcal{K}}(\partial)$ of 
${\mathcal{K}}(({\partial}^{-1}))$, generated by ${\mathcal{K}}[\partial]$, 
is called the skewfield of \emph{rational pseudodifferential operators} 
(see [CDSK12] for details). We have obvious inclusions :
$${\mathcal{K}} \subset {\mathcal{K}}[\partial] \subset {\mathcal{K}}(\partial) \subset{\mathcal{K}}(({\partial}^{-1})). $$
If the derivation acts trivially on $\mathcal{K}$, so that ${\mathcal{C}}={\mathcal{K}}$, letting ${\partial}=\lambda$, an indeterminate, commuting with 
elements of $\mathcal{K}$, we obtain inclusions of commutative algebras
$${\mathcal{C}} \subset {\mathcal{C}}[\lambda] \subset {\mathcal{C}}(\lambda) \subset{\mathcal{C}}(({\lambda}^{-1})).$$
It is well known that in many respects the non-commutative algebras
${\mathcal{K}}[\partial]$ and ${\mathcal{K}}(\partial)$ "behave" in a very 
similar way to that of ${\mathcal{C}}[\lambda]$ and ${\mathcal{C}}(\lambda)$. 
Namely, the ring ${\mathcal{K}}[\partial]$ 
is right (resp. left) Euclidean, hence any right (resp. left) ideal is 
principal. Moreover, any two right ideals $A{\mathcal{K}}[\partial]$ and 
$B{\mathcal{K}}[\partial]$ have non-zero intersection 
$M{\mathcal{K}}[\partial]$, where $M \neq 0$ is called the least right common 
multiple of $A$ and $B$; also $A{\mathcal{K}}[\partial] +B{\mathcal{K}}[\partial]=D{\mathcal{K}}[\partial]$, where $D$ is the greatest right common divisor 
of $A$ and $B$.
Furthermore, any element $H$ of ${\mathcal{K}}(\partial)$ has a right 
fractional decomposition $H=AB^{-1}$, where $B \neq 0$. A right fractional 
decomposition for which the differential operator $B$ has minimal order is 
called the minimal fractional decomposition (equivalently, the greatest 
common divisor of $A$ and $B$ is 1). It is unique up to multiplication of $A$ 
and $B$ on the right by the same non-zero element of ${\mathcal{K}}$. 
Any other fractional decomposition of $H$ is obtained from the minimal one by 
multiplication of $A$ and $B$ on the right by a non-zero element of 
${\mathcal{K}}[\partial]$. See [CDSK12] for details. 
Of course all these facts still hold if we replace "right" by "left".

\section{The Dieudonn\'e determinant}

The Dieudonn\'e determinant of an $n \times n$ matrix pseudodifferential 
operator $A \in M_n({\mathcal{K}}(({\partial}^{-1})))$ has the 
form $\det(A)={\det_1(A)}{{\lambda}^{d(A)}}$ where $\det_1(A) \in \mathcal{K}$, $\lambda$ is an indeterminate, and $d(A) \in \mathbb{Z}$. It exists and is 
uniquely defined by the following properties (see [Die43], [Art57]) :
\begin{enumerate}[$(i)$]
\item
$\det(AB)=\det(A)\det(B)$;
\item
If $A$ is upper triangular with non-zero diagonal entries 
$A_{ii} \in {\mathcal{K}}(({\partial}^{-1}))$ of degree (or order) $d(A_{ii})$ and 
leading coefficient $a_i \in \mathcal{K}$, then 
$$\det\nolimits_{1}(A)= \prod_{i=1}^{n} a_i, \:  d(A)= \sum_{i=1}^{n} d(A_{ii}).$$
By definition, $\det(A)=0$ if one of the $A_{ii}$ is $0$.
\end{enumerate}
Note that $\det_1(AB)=\det_1(A)\det_1(B)$.  
A matrix $A$ whose Dieudonn\'e determinant is non-zero is called 
\emph{non-degenerate}. In this case the integer $d(A)$ is well defined. It is 
called the degree of $det(A)$ and of $A$. Note that $d(AB)=d(A)+d(B)$ if both 
$A$ and $B$ are non-degenerate. 

\begin{lemma} 
\label{lem:3.1}
(a)Any $A \in M_n({\mathcal{K}}[\partial])$  
can be written in the form $A=UT$ (resp. $TU$), 
where $U$ 
is an invertible element of $M_n({\mathcal{K}}[\partial])$ 
and $T\in M_n({\mathcal{K}}[\partial])$ is upper triangular.

(b) Any non-degenerate $A \in 
M_n({\mathcal{K}}[\partial])$ 
can be written in the form $A=U_1DU_2$, where 
$ U_1, \, U_2$ are invertible elements of $M_n({\mathcal{K}}[\partial])$  
and $D$ is a diagonal $n\times n$ matrix with non-zero entries from 
${\mathcal{K}}[\partial]$. 
\end{lemma}
\begin{proof}
Recall that an elementary row (resp. column) operation of a matrix from 
$M_n({\mathcal{K}}[\partial])$ is either a permutation of 
two of its rows (resp. column), or adding to one row (resp. column) another 
one, multiplied on the left (resp. right) by an element of 
$\mathcal{K}[\partial]$. 
Since the row (resp. column) operations are equivalent to multiplication on 
the left (reps. right) by the corresponding elementary matrix, the first 
operation only changes the sign of the determinant and the second does not 
change it. 

In the proof of (a) we may assume that $A\neq 0$, and let $j$ be the minimimal 
index, for which the $j$-th column is non-zero. 
Among all matrices that can be obtained from $A$ by elementary row operations 
choose the one for which the $(1,j)$-entry is non-zero and has the minimal order. 
Then, by elementary row 
operations, using the Euclidean property of ${\mathcal{K}}[\partial]$, we 
obtain from $A$ a matrix $A_1$ such that all entries of the $j$-th column, 
except the first one, are zero. Repeating this process 
for the $(n-1) \times (n-1)$ submatrix obtained from $A_1$ by deleting the 
first row and column, we obtain the decomposition $A=UT$ as in (a).

For the decomposition $A=TU$, we use a similar argument, except that
we start from largest $j$ for which the $j$-th row is non-zero,
we perform column operations to have the $(j,n)$-entry non-zero and of minimal possible order,
and then we further make elementary column operations
to obtain a matrix $A_1$ such that all entries of the $j$-th row are zero, 
except the last one.
The claim follows by induction, after deleting the last row and column.

In order to obtain the decomposition in (b), we use the same argument, except 
that we choose among all matrices obtained from $A$ by elementary row and 
column operations the one for which the $(1,1)$-entry is non-zero and has the 
minimal order (it exists since $\det (A) \neq 0$). 
\end{proof}
\begin{corollary}\label{cor:3.2} 
Let $A \in M_n({\mathcal{K}}[\partial])$ be 
a non-degenerate matrix differential operator. Then 
\begin{enumerate}[(a)]
\item
$d(A) \in {\mathbb{Z}}_{+}$.
\item
$A$ is an invertible element of the ring 
$M_n({\mathcal{K}}[\partial])$ if and only if $d(A)=0$.
\end{enumerate}
\end{corollary}
\begin{remark}\label{rem:3.3}
Let $A \in M_n({\mathcal{K}}((\partial^{-1})))$ 
and let $A^*$ be the adjoint matrix pseudodifferential operator.
If $det(A)=0$, then $det(A^*)=0$.
If $det(A) \neq 0$, then $det(A^*)={(-1)}^{d(A)}det(A)$.
This follows from the obvious fact that $A$ can be brought by elementary row transformations
over the skewfield $\mc K((\partial^{-1}))$ to an upper triangular matrix,
and in this case the statement becomes clear.
\end{remark}
%

\section{Rational matrix pseudodifferential operators}

A matrix $H \in M_n({\mathcal{K}}(\partial))$ is called a 
\emph{rational matrix pseudodifferential operator}. In other words, all the 
entries of such a matrix have the form $h_{ij}={a_{ij}}{b_{ij}}^{-1}$, 
$i,j=1,...,n$, where $a_{ij}, b_{ij} \in {\mathcal{K}}[\partial]$ and all 
${b_{ij}} \neq 0$. 
Let $b(\neq 0)$ be the least right common multiple of the 
$b_{ij}$'s, so that $b_{ij}.c_{ij}=b$ for some ${c_{ij}} \neq 0$.  
Multiplying $a_{ij}$ 
and $b_{ij}$ on the right by $c_{ij}$, we obtain $H={A_1}b^{-1}$, where 
$(A_1)_{ij}={a_{ij}}{c_{ij}}$. In other words $H$ has the right fractional 
decomposition $H={A_1}(b \id_n)^{-1}$.
However, among all right fractional decompositions $H=AB^{-1}$, where 
$A,B \in M_n({\mathcal{K}}[\partial])$ and $\det B \neq 0$, this might be 
not the "best" one.

\begin{definition} A right fractional decomposition $H=AB^{-1}$, where $A,B \in 
M_n({\mathcal{K}}[\partial])$ and $det B \neq 0$, is called 
\emph{minimal} if $d(B)$ ( $ \in {\mathbb{Z}}_{+}$) is minimal among all right 
fractional decompositions of $H$. 
\end{definition}
Note that, if $H=AB^{-1}$ is a minimal fractional decomposition,
then $0 \leq d(B) \leq d(b)$, where 
$b$ is the least right common multiple of all the entries of $H$.

\begin{proposition}\label{prop:4.2}
Let $A$ and $B$ be two non-degenerate $n \times n$ matrix differential 
operators. Then one can find non-degenerate $n \times n$ matrix differential 
operators $C$ and $D$, such that $AC=BD$ ( resp. $CA=DB$ ) 
\end{proposition}
\begin{proof} 
By induction on $n$. We know it is true in the scalar case, 
see e.g. \cite{CDSK12}. 
By Lemma \ref{lem:3.1}, 
multiplying on the right by
invertible matrices, we may assume that both $A$ and $B$ 
are upper triangular matrices. 
Let
$$
A=\left(
\begin{array}{ccc|c}
 & & &        \\
A_1& & & U \\
 & & & \\ \hline
0& & & a\\
\end{array}\right),\,\,\,\,\,
B=\left(
\begin{array}{ccc|c}
 & & &        \\
B_1& & &V \\
 & & & \\ \hline
0& & & b\\
\end{array}\right),
$$
where $A_1,B_1\in M_{n-1}(\mc K[\partial])$ are upper triangular non-degenerate,
$U,V\in\mc K[\partial]^n$, and $a,b\in\mc K[\partial]\backslash\{0\}$.
By the inductive assumption, there exist $C_1,D_1\in M_{n-1}(\mc K[\partial])$ non-degenerate,
such that $A_1C_1=B_1D_1$,
and $c,d\in\mc K[\partial]\backslash\{0\}$ such that $ac=bd$.
Hence, after multiplying on the right $A$ by the block diagonal matrix with $C_1$ and $c$ on the diagonal,
and $B$ by the block diagonal matrix with $D_1$ and $d$ on the diagonal,
we may assume that
$$
A_1=B_1
\,\,\text{ and }\,\,
a=b
\,.
$$
Consider the matrix 
$$
M=\left(
\begin{array}{ccc|c}
 & & &        \\
A_1 & & &U-V \\
 & & & \\ \hline
0& & & 0\\
\end{array}\right)\in M_n(\mc K[\partial])\,.
$$
Viewed over the skewfield ${\mathcal{K}}(\partial)$, it has a non-zero kernel (since 
$M: \mathcal{K}(\partial)^n \mapsto \mathcal{K}(\partial)^n$ is not 
surjective), 
i.e. there exists a vector $\tilde X=\left(\begin{array}{c} X \\ x \end{array}\right) 
\in {{\mathcal{K}}(\partial)}^n$,
where $X\in {{\mathcal{K}}(\partial)}^{n-1}$ and $x\in\mc K(\partial)$,
such that $M\tilde X=0$,
i.e. 
\begin{equation}\label{eq:pro42}
A_1X+Ux=Vx
\,.
\end{equation}
Replacing $\tilde X$ by $\tilde Xd$, where 
$d$ is a non-zero common multiple of all the denominators of the entries of $\tilde X$,
we may assume that $\tilde X \in {{\mathcal{K}}[\partial]}^n$. 
Note also that, since $A_1$ is non-degenerate, it must be $x \neq 0$. 
To conclude the proof we just observe that, by \eqref{eq:pro42},
we have the identity $AE=BF$, where
$$
E=\left(
\begin{array}{ccc|c}
 & & &        \\
\id_{n-1}& & & X \\
 & & & \\ \hline
0& & & x\\
\end{array}\right), \,\,\,\,\,
F=\left(
\begin{array}{ccc|c}
 & & &        \\
\id_{n-1}& & &0 \\
 & & & \\ \hline
0& & & x\\
\end{array}\right).
$$
\end{proof}

\begin{remark} Proposition \ref{prop:4.2} can be derived from Goldie theory (see 
\cite[Theorem 2.1.12]{MR01}), but we opted for a simple direct argument.
\end{remark}
\begin{theorem}\label{20120119:cl1}
For every matrix differential operators $A,B\in M_n(\mathcal{K}[\partial])$
with $\det (B)\neq0$,
there exist matrices $A_1,B_1,D\in M_n({\mathcal{K}}[\partial])$,
with $\det {B_1}\neq0,\det D\neq0$,
such that:
\begin{enumerate}[$(i)$]
\item
$A=A_1D$, $B=B_1D$,
\item
$\ker A_1\cap\ker B_1=0$.
\end{enumerate}
\end{theorem}
\begin{proof}
We will prove the statement by induction on $d(B)$. 
If $d(B)=0$, then $B$ is invertible in $M_n(\mathcal{K}[\partial])$ by 
Corollary \ref{cor:3.2}
(and $\ker B=0$). In this case the claim holds trivially, taking $D=\id_n$. 
Clearly, if $P \in M_n(\mathcal{K}[\partial])$ is invertible,
then $\ker A=\ker PA$.
Hence, if $P$ and $Q$ are invertible elements 
of $M_n(\mathcal{K}[\partial])$,
then the statement holds for $A$ and $B$ if and only if it holds for $PA$ 
and $QB$.
Furthermore, 
if $R\in M_n(\mathcal{K}[\partial])$ is invertible,
replacing $D$ by $R^{-1}D$ we get that
the statement holds for $A$ and $B$ if and only if it holds for $AR$ and $BR$.
Therefore, by Lemma \ref{lem:3.1}, we may assume, without loss of generality, that
$A$ is upper triangular and $B$ is diagonal.
If $\ker A\cap\ker B=0$ there is nothing to prove.
Let then $F=\big(f_i\big)_{i=1}^n$ be a non-zero element of $\ker A\cap\ker B$,
and let $k\in \{1,\dots,n\}$ be such that $f_k \neq 0$, $f_{k+1}=\dots=f_n=0$.
The condition $AF=0$ gives for $i=1,\dots,k$,
$$
A_{i,1}(\partial)f_1+\dots+A_{i,k-1}(\partial)f_{k-1}+A_{ik}(\partial)f_k=0
\quad \text{ in } \mathcal{K}.
$$ 
This implies that there is some $L_i(\partial)\in {\mathcal{K}}[\partial]$ 
such that
\begin{equation}\label{20120124:eq1}
A_{i,1}(\partial)\circ\frac{f_1}{f_k}+\dots+A_{i,k-1}(\partial)\circ\frac{f_{k-1}}{f_k}
+A_{ik}(\partial)=L_i(\partial)\circ\Big(\partial-\frac{f_k'}{f_k}\Big)
\quad \text{ in }  {\mathcal{K}}[\partial]\,.
\end{equation} 
Indeed, the LHS above is zero when applied to $f_k\in {\mathcal{K}}$,
hence it must be divisible, on the right, by $\partial-\frac{f_k'}{f_k}$.
Similarly, from the condition $BF=0$ we have that
$B_{ii}(\partial)f_i=0$ in $\mc {\mathcal{K}}$ for every $i=1,\dots,k$,
which implies that there is some $M_i(\partial)\in\mc {\mathcal{K}}[\partial]$ such that
\begin{equation}\label{20120124:eq2}
B_{ii}(\partial)\circ\frac{f_i}{f_k}=M_i(\partial)\circ\Big(\partial-\frac{f_k'}{f_k}\Big)
\quad \text{ in } {\mathcal{K}}[\partial]\,.
\end{equation}
Let then $A_1,B_1,D\in M_n({\mathcal{K}}[\partial])$
be the matrices defined as 
the matrices $A,B,\id$ with the $k$-th column replaced, respectively, by the 
following columns
$$
\left[\begin{array}{c} L_1 \\ \vdots  \\ L_{k-1} \\ L_k \\ 0 \\ \vdots \\ 0 \end{array}\right]
\,\,,\,\,\,\,
\left[\begin{array}{c} M_1 \\ \vdots  \\ M_{k-1} \\ M_k \\ 0 \\ \vdots \\ 0 \end{array}\right]
\,\,,\,\,\,\,
\left[\begin{array}{c} -f_1/f_k \\ \vdots  \\ -f_{k-1}/f_k \\ \partial-f_k'/f_k \\ 0 \\ \vdots \\ 0 \end{array}\right]
\,.
$$
It follows from equations (\ref{20120124:eq1}) and (\ref{20120124:eq2})
that $A_1D=A$ and $B_1D=B$.
Moreover, since $\det D=\lambda$, we have $d(B_1)=d(B)-1$.
The statement follows by the inductive assumption.
\end{proof}
%

%
%

\section{Linear closure of a differential field}

In this section we define a natural embedding of a differential field in a linearly closed one using the theory of Picard-Vessiot extensions. One may find 
all relevant definitions and constructions in Chapter 3 of \cite{Mag94}.

Recall \cite{DSK11} that a differential field $\mc K$ is called \emph{linearly closed}
if every homogeneous linear differential equation of order $n\geq1$,
\begin{equation}\label{20120121:eq1}
a_nu^{(n)}+\dots+a_1u'+a_0u=0\,,
\end{equation}
with $a_0,\dots,a_n$ in $\mc K$, $a_n\neq0$,
has a non-zero solution $u\in\mc K$.

It is easy to show that
the solutions of equation \eqref{20120121:eq1} in a differential field $\mc K$
form a vector space over the field of constant $\mc C$
of dimension less than or equal to $n$,
and equal to $n$ if $\mc K$ is linearly closed (see e.g. \cite{DSK11}).
\begin{remark}
In a linearly closed field, it is also true that every inhomogeneous linear
differential equation $L(\partial)u=b$ has a solution because the homogeneous 
differential equation $((1/b)L(\partial)u)'=0$ has a solution $u$ such that 
$L(\partial)u \neq 0$ (the solutions of $((1/b)L(\partial)u)'=0$ form a vector 
space over the subfield of constants $\mc C$ of dimension strictly bigger than the one of 
$Ker L$).
\end{remark}
More generally, if $A\in M_{n}(\mc K[\partial])$
is a non-degenerate matrix differential operator 
and $b\in\mc K^n$,
then the inhomogeneous system of linear differential equations
in $u=\big(u_i\big)_{i=1}^n$,
\begin{equation}\label{20120121:eq5}
A(\partial)u=b\,,
\end{equation}
admits the affine space (over $\mc C$) of solutions
of dimension less than or equal to $d (A)$,
and equal to $d(A)$ if $\mc K$ is linearly closed. (This follows, for example,
from Lemma \ref{lem:3.1}(b).)

\begin{definition}
Let $\mc K$ be a differential field with the subfield of constants 
$\mc C$, and let $L\in \mc K[\partial]$ be a differential operator over 
$\mc K$ of order $n$. 
A differential field  extension $\mc K  \subset  \mc L$ is called
a Picard-Vessiot extension with respect to $L$ if there are no new constants 
in $\mc L$ and if $\mc L=\mc K (y_1,...,y_n)$, where the $y_i$ are 
linearly independent solutions over $\mc C$ of the equation $Ly=0$.
\end{definition}
Proofs of the following two propositions can be found in \cite{Mag94}.
\begin{proposition}\label{prop:5.3}
Let $\mc K$ be a differential field with algebraically closed subfield of constants $\mc C$ and let $L$ be a differential operator of order $n$ over $\mc K$.
Then there exists a Picard-Vessiot extension of $\mc K$ with respect to $L$ 
and it is unique up to isomorphism.
\end{proposition}

\begin{proposition}\label{prop:5.4}
If $\mc K \subset \mc L$ is an extension of differential fields and $\mc K \subset \mc {E_i} \subset \mc L$, $i=1,2$, are two Picard-Vessiot subextensions of $\mc K$, 
then the composite field ${{\mc E}_1}{{\mc E}_2}$ 
(i.e. the minimal subfield of $\mc L$ containing both $\mc E_1$ and $\mc E_2$)
is a Picard-Vessiot extension of $\mc K$ as well.
\end{proposition}

\begin{definition}
Let $\mc K$ be a differential field with algebraically closed subfield of 
constants $\mc C$. The unique minimal extension $\mc K \subset \mc L$ such that
\begin{enumerate}[(a)]
\item
$\mc L$ is the union of its Picard-Vessiot subextensions of $\mc K$;
\item
$\mc L$ contains an isomorphic copy of every Picard-vessiot extension 
of $\mc K$,
\end{enumerate}
is called the Picard-Vessiot compositum of $\mc K$.
\end{definition}

It is proved in \cite{Mag94} that the Picard-Vessiot compositum of $\mc K$
exists, and is unique up to isomorphism.

\begin{definition}
Let $\mc K$ be a differential field with algebraically closed subfield of 
constants. Let ${\mc K}_0=\mc K$ and, for $i\in \mathbb{Z}_+$, let 
${\mc K}_{i+1}$ be the Picard-Vessiot compositum of ${\mc K}_i$. 
We call $\mc L= \cup_i {\mc K}_i$ the \emph{linear closure} 
of $\mc K$ (it is called the successive Picard-Vessiot closure in \cite{Mag94} ).
\end{definition}
\begin{remark}
The linear closure is linearly closed.
\end{remark}
\begin{remark}
The linear closure of a differential field ${\mc K}$ with algebraically 
closed subfield of constants is the unique, up to isomorphism, minimal linearly closed extension of $\mc K$ with no new constants. To see this, one needs 
to show that for any linearly closed extension $\mc L$ of $\mc K$ without 
new constants, one can extend the embedding 
$\mc K \hookrightarrow \mc L$ to an embedding of the Picard-Vessiot 
compositum of $\mc K$, ${\mc K}_1 \hookrightarrow \mc L$. By Zorn's lemma one 
can find a maximal subextension 
$\mc K \subset \tilde{\mc K} \subset {\mc K}_1$ extending the embedding 
$\mc K \hookrightarrow \mc L$. 
Denote by $\phi$ the embedding $\tilde{\mc K} \hookrightarrow \mc L$. 
Suppose that $\tilde{\mc K} \subsetneq {\mc K}_1$. This means that, by definition 
of ${\mc K}_1$, we have a non-trivial Picard-Vessiot extension 
$\tilde{\mc K} \subset \mc P \subset {\mc K}_1$ for a differential operator
$L$ over $\mc K$. As $\mc L$ is linearly closed, we can find a Picard-Vessiot 
extension $\phi(\tilde{K}) \subset {\mc P}_1 \subset \mc L$ for the same 
differential operator. By Proposition \ref{prop:5.3}, these two Picard-Vessiot 
extension are isomorphic and one can extend the embedding
$\tilde{\mc K} \hookrightarrow \mc L$ to an embedding
$\mc P \hookrightarrow \mc L$, which is a contradiction. 
\end{remark}
\begin{lemma}\label{lem:5.9}
Let $\mc K$ be a differential field with algebraically closed subfield of 
constants, let $\mc L$ be its linear closure, and let $X$ be a finite subset 
of $\mc L$, not contained in $\mc K$. Then there is an integer $i$ and a 
Picard-Vessiot extension ${\mc K}_i \subset \mc P  \subset {\mc K}_{i+1}$ 
of $\mc K_i$ such that $X \subset \mc P$ but $X \not\subset {\mc K}_i$.
\end{lemma}

\begin{proof} 
Take the minimal $i$, such that $X \subset {\mc K}_{i+1}$.
Since $\mc K_{i+1}$ is the Picard-Vessiot compositum of $\mc K_i$,
every element of $X$ lies in a Picard-Vessiot extention of $\mc K_i$.
The claim follows by the fact that 
the composite of two Picard-Vessiot extension is still a Picard-Vessiot extension (Proposition \ref{prop:5.4}).
\end{proof}

\begin{lemma}\label{20120121:cl3}
Let $\mc K\subset\mc L$ be a differential field extension,
and let $\mc C\subset\mc D$ be the corresponding field extension of constants.
If $\alpha\in\mc L$ is algebraic over $\mc C$,
then $\alpha\in\mc D$ and the minimal monic polynomial for $\alpha$
over $\mc K$ has coefficients in $\mc C$.
\end{lemma}
\begin{proof}
Let $P(x)=x^n+c_1x^{n-1}+\dots+c_n\in\mc C[x]$ be the minimal monic polynomial
with coefficients in $\mc C$ satisfied by $\alpha$.
Letting $x=\alpha$ and applying the derivative $\partial$ we get
$\big(n\alpha^{n-1}+(n-1)c_1\alpha^{n-2}+\dots+c_n\big)\alpha'=0$.
By minimality of $P(x)$, it must be $\alpha'=0$, i.e. $\alpha\in\mc D$.

Similarly, for the second statement, 
let $Q(x)=x^m+f_1x^{m-1}+\dots+f_m\in\mc K[x]$ be the minimal monic polynomial
with coefficients in $\mc K$ satisfied by $\alpha$.
Letting $x=\alpha$ and applying the derivative $\partial$ we get
$f_1'\alpha^{m-1}+\dots+f_m'=0$,
which, by minimality of $Q(x)$, implies $f_1,\dots,f_m\in\mc C$.
\end{proof}

\begin{lemma}[see e.g. \cite{PS03}]\label{20120121:cl2}
Let $\mc K$ be a differential field with subfield of constants $\mc C$.
Then elements $f_1,\dots,f_n\in\mc K$ are linearly independent 
over any subfield of $\mc C$ if and only if their Wronskian is non-zero.
\end{lemma}

\begin{lemma}\label{20120121cl1}
\begin{enumerate}[(a)]
\item
Let $\mc K$ be a differential field with field of constants $\mc C$,
and let $\mc D$ be an algebraic extension of $\mc C$.
Then $\mc D\otimes_{\mc C}\mc K$ is a differential field
with field of constants $\mc D$.
\item
Let $\mc K$ be a differential field with field of constants $\mc C$,
and let ${\mc L}$ be a differential field extension of $\mc K$ 
with field of constants $\bar{\mc C}$, the algebraic closure of $\mc C$.
Then, for every algebraic extension $\mc D$ of $\mc C$, 
the differential field $\mc D\otimes_{\mc C}\mc K$ is canonically isomorphic
to a differential subfield of ${\mc L}$.
\end{enumerate}
\end{lemma}
\begin{proof}
For part (a) we need to prove that every non-zero element
$f=\sum_ic_i\otimes f_i\in\mc D\otimes_{\mc C}\mc K$
is invertible.
Let $\mc C[\alpha]$ be a finite extension of $\mc C$ in $\mc D$
containing all elements $c_1,\dots,c_n$,
and let $P(x)\in\mc C[x]$ be the minimal monic polynomial for $\alpha$ over $\mc C$.
By Lemma \ref{20120121:cl3}, $P(x)$ is an irreducible element of $\mc K[x]$.
Therefore $\mc K[x]/(P(x))$ is a field,
and $f\in\mc C[\alpha]\otimes_{\mc C}\mc K\simeq\mc K[x]/(P(x))$
is invertible.

Next, we prove part (b).
By the universal property of the tensor product,
there is a canonical map 
$\varphi:\,\mc D\otimes_{\mc C}\mc K\to\mc L$
given by $\varphi(c\otimes f)=cf$.
This is a differential field embedding by part (a).
\end{proof}

\begin{definition}
Let $\mc K$ be a differential field with subfield of constants $\mc C$. We 
know from Lemma \ref{20120121cl1}(a) that ${\bar{\mc C}} \otimes_{\mc C} {\mc K}$ is a d
differential field with subfield of constants $\bar{\mc C}$. We define the
linear closure of $\mc K$ to be the one of 
${\bar{\mc C}} \otimes_{\mc C} {\mc K}$.
\end{definition}

Recall that the \emph{differential Galois group} $Gal(\mc L/\mc K)$
of a differential field extension $\mc K\subset\mc L$
is defined as the group of automorphisms of $\mc L$ commuting with $\partial$
and fixing $\mc K$.
One of the main properties of Picard-Vessiot extensions
is the following
\begin{proposition}[\cite{PS03}]\label{20120123:prop1}
Let $\mc K$ be a differential field with algebraically closed subfield of constants $\mc C$,
and let $\mc L$ be a Picard-Vessiot extension of $\mc K$.
Then,
the set of fixed points of the differential Galois group $Gal(\mc L/\mc K)$ is $\mc K$.
\end{proposition}

\section{Minimal fractional decomposition}
\label{sec:app3}

Given a matrix $A\in M_n({\mathcal{K}}[\partial])$,
we denote by $\bar A$ the same matrix $A$ considered as an endomorphism
of $\bar{\mathcal{K}}^n$, where $\bar{\mathcal{K}}$ is the linear closure of 
$\mathcal{K}$.
We have the following possible conditions for a ``minimal'' fractional decomposition $H=AB^{-1} \in M_n({\mathcal{K}}(\partial))$, 
where $A, B \in M_n({\mathcal{K}}[\partial])$ and $B$ is non-degenerate:
\begin{enumerate}[$(i)$]
\item
$d(B)$ is minimal among all possible fractional decompositions of $H$;
\item
$A$ and $B$ are \emph{coprime}, i.e.
if $A=A_1D$ and $B=B_1D$, with $A_1,B_1,D\in M_n({\mathcal{K}}[\partial])$,
then $D$ is invertible in $M_n({\mathcal{K}}[\partial])$;
\item
$\ker\bar A\cap\ker\bar B=0$.
\end{enumerate}
Obviously, condition $(iii)$ implies:
\begin{enumerate}
\item[$(iii')$]
$\ker A\cap\ker B=0$.
\end{enumerate}
\begin{example}\label{20120119:cl9}
Condition $(iii')$ is weaker than condition $(iii)$.
Consider, for example, $A=\partial(\partial-1)$ and $B=\partial-1$.
We have $e^x\in\ker\bar A\cap\bar B$,
and $\ker A\cap\ker B=0$ unless the differential field 
$\mc K$
contains a solution to the equation $u'=u$.
\end{example}
\begin{remark}
Condition (iii) is equivalent to ask that $A$ and $B$ have no common 
eigenvector with eigenvalue $0$ over any differential field extension of 
$\mc K$.
\end{remark}
\begin{proposition}\label{20120126:prop1}
In the ``scalar'' case $n=1$, conditions $(i)$, $(ii)$ and $(iii)$ are equivalent.
\end{proposition}
\begin{proof}
It follows from \cite{CDSK12} that conditions $(i)$ and $(ii)$ are equivalent.
Moreover, condition $(iii)$ implies condition $(ii)$ since,
if $D\in\mc K[\partial]$ is not invertible,
than it has some root in the linear closure $\bar {\mc K}$.
We are left to prove that condition $(ii)$ implies condition $(iii)$.
Note that, by the
Euclidean algorithm, the right greatest common divisor of $A$ and $B$
is independent of the differential field extension of $\mc K$.
Suppose, by contradiction, that 
$0\neq f\in\ker\bar A\cap\ker\bar B$,
which means that $A=A_1(\partial-\frac{f'}{f})$
and $B=B_1(\partial-\frac{f'}{f})$,
for some $A_1, B_1\in\bar{\mc K}[\partial]$,
so that the right greatest common divisor of $A$ and $B$ is not invertible,
contradicting assumption $(ii)$.
\end{proof}
\begin{theorem}
\label{th:6.4}
\begin{enumerate}[(a)]
\item
Every $H \in M_n({\mathcal{K}}(\partial))$ can be represented as $H=AB^{-1}$, with $B$ non-degenerate, such that $(iii)$ holds.
\item
Conditions $(i)$, $(ii)$, and $(iii)$ are equivalent. Any fraction which 
satisfies one of these equivalent conditions is called a minimal fractional 
decomposition.
\item
If ${A_0}{B_0}^{-1}$ is a minimal fractional decomposition of the fraction 
$H=AB^{-1}$, then one can find a matrix differential operator $D$ such that 
$A={A_0}D$ and $B={B_0}D$.
\end{enumerate}
\end{theorem}
\begin{proposition}\label{prop:6.5}
Theorem \ref{th:6.4} holds if $\mc K$ is linearly closed. 
\end{proposition}
\begin{lemma}\label{lem:alberto1}
Assuming that Theorem \ref{th:6.4}(c) holds,
let $K=A_1B_1^{-1}$ be a minimal fractional decomposition, with $A_1,B_1\in M_k(\mc K[\partial])$.
Let also $V\in\mc K[\partial]^k$ be such that $AB^{-1}V\in\mc V[\partial]^\ell$.
Then $V=BZ$ for some $Z\in\mc K[\partial]^k$.
\end{lemma}
\begin{proof}
After replacing, if necessary, $A$ by $AU_1$, $B$ by $U_2BU_1$, and $V$ by $U_2V$,
with $U_1$ and $U_2$ invertible elements of $M_n(\mc K[\partial])$,
we can assume by Lemma \ref{lem:3.1} that $B$ is diagonal.
If $V=0$ there is nothing to prove, so let the $i$-th entry of $V$ be non zero.
Consider the matrix $\widetilde{V}\in M_k(\mc K[\partial])$
be the same as $B$, with the $i$-th column replaced by $V$.
Clearly, $\widetilde{V}$ is non-degenerate.
By assumption $AB^{-1}\widetilde{V}=K$ lies in $M_n(\mc K[\partial])$,
so that $K\widetilde{V}^{-1}$ is another fractional decomposition for $H=AB^{-1}$.
Hence, by Thorem \ref{th:6.4}(c), we have that $\widetilde{V}=B\widetilde{Z}$
for some $\widetilde{Z}\in M_n(\mc K[\partial])$,
so that $V=BZ$, where $Z$ is the $i$-th column of $\widetilde{Z}$.
\end{proof}

\begin{proof}[Proof of Proposition \ref{prop:6.5}]
Part $(a)$ holds by Theorem \ref{20120119:cl1}. 
In part $(b)$, 
condition $(iii)$ implies condition $(ii)$ since, by assumption, $\mc K$ is linearly closed.
Conversely, let $A,B\in M_n(\mc K[\partial])$ satisfy condition $(ii)$.
By Theorem \ref{20120119:cl1} we have $A=A_1D, B=B_1D$
with $\ker A_1\cap\ker B_1=0$,
and by assumption $(ii)$, $D\in M_n(\mc K[\partial])$ is invertible.
Hence,  $\ker A\cap\ker B=0$, proving $(iii)$.
Furthermore, it is clear that condition $(i)$ implies condition $(iii)$. 
Indeed if $Ker A \cap Ker B \neq 0$, then by 
Theorem \ref{20120119:cl1} one can find $C,D,E$ such that $A=CE$, $B=DE$, 
$Ker C \cap Ker D =0$ and $d(E)>0$. 
Then $AB^{-1}=CD^{-1}$ and $d(D)<d(B)$, contradicting assumption $(i)$.
To conclude, we are going to prove, by induction on $n$, that condition $(iii)$ implies condition $(i)$,
and that part $(c)$ holds.

If $n=1$ the statement holds by Proposition \ref{20120126:prop1} and the results in \cite{CDSK12}.
Let then $n>1$ and $A,B\in M_n(\mc K[\partial])$, with $B$ non degenerate, 
be such that condition $(iii)$ holds: $\ker A\cap\ker B=0$.
Let also $CD^{-1}=AB^{-1}$ be any other fractional decomposition of $H=AB^{-1}$, 
with $C,D\in M_n(\mc K[\partial])$, $D$ non degenerate.
We need to prove that there exists $T\in M_n(\mc K[\partial])$ such that
$C=AT$ and $D=BT$.
(In this case, $d(D)=d(B)+d(T)\geq d(B)$, proving condition $(i)$).

Firts, note that, if $U_i$, $i=1,\dots,4$, are invertible elements of $M_n(\mc K[\partial])$,
then $\ker(U_1AU_3)\cap\ker(U_2BU_3)=0$, 
and we have $(U_1AU_3)(U_2BU_3)^{-1}=(U_1CU_4)(U_2DU_4)^{-1}$.
Hence, by Lemma \ref{lem:3.1} we can assume, without loss of generality,
that $B$ is diagonal, $A,D$ are upper triangular,
and hence $C=AB^{-1}D$ is upper triangula as well.
Let then
$$
\begin{array}{l}
\displaystyle{
A=\left(
\begin{array}{ccc|c}
 & & &        \\
A_1& & &U \\
 & & & \\ \hline
0& & & a\\
\end{array}\right)
\,\,,\,\,\,\,
B=\left(
\begin{array}{ccc|c}
 & & &        \\
B_1& & &0 \\
 & & & \\ \hline
0& & & b\\
\end{array}\right)
\,,}\\
\displaystyle{
C=\left(
\begin{array}{ccc|c}
 & & &        \\
C_1& & &V \\
 & & & \\ \hline
0& & & c\\
\end{array}\right)
\,\,,\,\,\,\,
D=\left(
\begin{array}{ccc|c}
 & & &        \\
D_1& & &W \\
 & & & \\ \hline
0& & & d\\
\end{array}\right)
\,,}
\end{array}
$$
where $B_1$ is diagonal and $A_1,C_1,D_1$ are upper triangular $n-1\times n-1$ matrices  
with entries in $\mc K[\partial]$, 
with $B_1$ and $D_1$ non degenerate,
$U,V,W$ lie in $\mc K[\partial]^{n-1}$,
and $a,b,c,d$ lie in $\mc K[\partial]$, with $b,d\neq0$.
By assumption $AB^{-1}=CD^{-1}$, meaning that
\begin{equation}\label{eq:alb1}
A_1B_1^{-1}=C_1D_1^{-1}
\,\,,\,\,\,\,
ab^{-1}=cd^{-1}
\,\,,\,\,\,\,
Ub^{-1}=-C_1D_1^{-1}Wd^{-1}+Vd^{-1}\,.
\end{equation}
Moreover, the assumption $\ker A\cap\ker B=0$
clearly implies that $\ker A_1\cap\ker B_1=0$
(if $X\in \mc K^{n-1}$ is such that $A_1(X)=B_1(X)=0$, 
then $\tilde X=\left(\begin{array}{c}X\\0\end{array}\right)\in\mc K^n$ lies in $\ker A\cap\ker B$).
Hence, by the first identity in \eqref{eq:alb1} and
the inductive assumption, 
there exists $T_1\in M_{n-1}(\mc K[\partial])$ such that 
\begin{equation}\label{eq:t1}
C_1=A_1T_1
\,\,,\,\,\,\,
D_1=B_1T_1\,.
\end{equation}
The main problem is that we do not know that $\ker a\cap\ker b=0$
(it is false in general), hence we cannot conclude, yet, that $c=at$ and $d=bt$
for some $t\in\mc K[\partial]$.
Let then $ef^{-1}$ be a minimal fractional decomposition of $ab^{-1}=cd^{-1}$.
By the $n=1$ case we know that there exist $p,q\in\mc K[\partial]$
such that
\begin{equation}\label{eq:t2}
a=ep
\,\,,\,\,\,\,
b=fp
\,\,,\,\,\,\,
c=eq
\,\,,\,\,\,\,
d=fq\,,
\end{equation}
and let $k\in\mc K[\partial]$ be a right greatest common divisor of $p$ ad $q$,
i.e. there exist $s,t,i,j\in\mc K[\partial]$ such that
\begin{equation}\label{eq:t3}
p=ks
\,\,,\,\,\,\,
q=kt
\,\,,\,\,\,\,
si+tj=1\,.
\end{equation}
Eventually we will want to prove that we can choose
$k=p$ (i.e. $s=1$, $i=1$ and $j=0$).
Using the identities \eqref{eq:t1}, \eqref{eq:t2} and \eqref{eq:t3},
we can rewrite the third equation in \eqref{eq:alb1} as follows
$$
Us^{-1}=-A_1B_1^{-1}Wt^{-1}+Vt^{-1}\,,
$$
and multiplying each side of the above equation
by each side of the identity $1-si=tj$, we get
\begin{equation}\label{eq:alb2}
U+A_1B_1^{-1}Wjs=(Ui+Vj)s\,.
\end{equation}
Since $A_1B_1^{-1}$ is a minimal fractional decomposition,
we get, by the inductive assumption and Lemma \ref{lem:alberto1},
that there exists $Z\in\mc K[\partial]^{n-1}$ such that
\begin{equation}\label{eq:alb3}
Wjs=B_1Z
\,\,,\,\,\,\,
U+A_1Z=(Ui+Vj)s\,.
\end{equation}
%
Let $x\in\mc K$ be such that $s(x)=0$.
Letting $X=\left(\begin{array}{c}Z(x) \\ x\end{array}\right)\in\mc K^n$,
we get
$$
\begin{array}{l}
\displaystyle{
A(X)
=\left(\begin{array}{c} A_1Z(x)+U(x) \\ a(x) \end{array}\right)
=\left(\begin{array}{c} (Ui+Vj)s(x) \\ eks(x) \end{array}\right)=0
\,,}\\
\displaystyle{
B(X)
=\left(\begin{array}{c} B_1Z(x) \\ b(x) \end{array}\right)
=\left(\begin{array}{c} Wjs(x) \\ fks(x) \end{array}\right)=0\,.
}
\end{array}
$$
Hence, since by assumption $\ker A\cap\ker B=0$,
it follows that $\ker s=0$.
Namely, since $\mc K$ is linearly closed, $s$ is a scalar, that we can shoose to be $1$.
In conclusion, we get, as we wanted, that $k=p$ and $q=pt$,
so that, by \eqref{eq:t2},
\begin{equation}\label{eq:t4}
c=at\,\,,\,\,\,\,d=bt\,.
\end{equation}

Going back to the third equation in \eqref{eq:alb1}, we then get
\begin{equation}\label{eq:alb4}
Ut=-A_1B_1^{-1}W+V\,.
\end{equation}
Again, by the inductive assumption on the minimality of $A_1B_1^{-1}$
and Lemma \ref{lem:alberto1},
it follows that there exists $Z\in\mc K[\partial]^{n-1}$ such that
\begin{equation}\label{eq:alb5}
W=B_1Z
\,\,,\,\,\,\,
V=Ut+A_1Z\,.
\end{equation}
Hence, letting
$$
T=\left(
\begin{array}{ccc|c}
 & & &        \\
T_1& & &Z \\
 & & & \\ \hline
0& & & t\\
\end{array}\right)\,\in M_n(\mc K[\partial])\,,
$$
we get that $C=AT$ and $D=BT$, completing the proof.
\end{proof}
\begin{proposition}\label{prop:6.6}
Part $(a)$ of Theorem \ref{th:6.4} holds, namely if $A$ and $B$ are two 
$n\times n$ matrix differential operators with $B$ non-degenerate,
then we can find $n \times n$ matrix differential operators $C$, $D$ and $E$, 
such that $A=CE$, $B=DE$ and $Ker \bar{C} \cap Ker \bar{D}=0$. 
\end{proposition}
\begin{proof}
First, assume that the subfield of constants of $\mc K$ is algebraically 
closed. By Lemma \ref{lem:3.1}, we may assume that is $A$ upper triangular and 
$B$ is diagonal. Consider a minimal fractional decomposition $CD^{-1}$ of the 
fraction $AB^{-1}$ in the linear closure of $\mc K$. By Lemma \ref{lem:3.1}(a),
we can choose $C$ and $D$ to be upper triangular matrix differential operators.
We may assume that all the diagonal 
entries of $D$ are monic and, using elementary column transformations,
that $d(D_{ij})<d(D_{ii})$ for all $i<j$. 
Since the linear closure is the union of the iterate Picard-Vessiot 
compositum of $\mc K$, all the coefficients of the entries of $C$ and of $D$ 
lie in some iterate Picard-Vessiot compositum of $\mc K$. Take $i$ minimal 
such that ${\mc K}_i$ satisfies this property. Assume $i \neq 0$. 
By Lemma \ref{lem:5.9}, all the coefficients of the entries of $C$ and $D$ lie in 
some Picard-Vessiot subextension ${\mc K}_{i-1} \subset \mc P \subset {\mc K}_i$. Pick an automorphism $\phi$ of this extension. By Theorem \ref{th:6.4}, 
the fractional decomposition ${\phi(C)}{\phi(D)}^{-1}=CD^{-1}$ is still a 
minimal fractional decomposition because $d({\phi(D)})=d(D)$. So $C$ 
(resp $D$) and $\phi(C)$ (resp $\phi(D)$) are equal up to right 
multiplication by an invertible upper triangular matrix differential operator 
$E$. As all the diagonal entries of $\phi(D)$ are monic and 
$deg({\phi(D)}_{ij})<deg({\phi(D)}_{ii})$ for all $i<j$, $E$ has to be the 
identity matrix. Hence $C=\phi(C)$ and $D=\phi(D)$ for all $\phi$. 
It follows, by Proposition \ref{20120123:prop1},
that all the coefficients of the entries of $C$ and $D$ actually lie in 
${\mc K}_{i-1}$, which is a contradiction. 
So $i=0$ and all the coefficients of $C$ and $D$ are differential operators 
over $\mc K$. 

In the general case, one can find $C$, $D$ and $E$ satisfying the assumptions
of the proposition, whose entries are differential operators a priori over ${\mathcal{K}}\otimes_{\mc C } \bar{\mathcal{C}}$. So all the coefficients of the 
entries of $C$ and $D$ lie in a Galois extension $\mc K \subset \mc G$. 
As the extension of the derivation to an algebraic extension is unique, 
all automorphisms commute with the derivation. Hence, using the same argument 
as above with the usual Galois theory, we obtain that the entries of $C$ and 
$D$, hence those of $E$, are actually differential operators over $\mc K$.
\end{proof}

\begin{proof}[Proof of Theorem \ref{th:6.4}] By Proposition \ref{prop:6.6},
condition $(i)$ implies condition $(iii)$. Let $AB^{-1}$ be a 
fractional decomposition, satisfying $(iii)$. Then by Proposition \ref{prop:6.5} 
it satisfies $(i)$ as a fraction of matrix differential 
operators over the 
linear closure of $\mc K$, hence a fortiori over $\mc K$. 
The implication $(iii) \Rightarrow (ii)$ is clear by definition of a linearly 
closed field and $(ii) \Rightarrow (iii)$ follows from Proposition  \ref{prop:6.6}. 
Hence part $(b)$ of the theorem holds. If $A_{0}B_{0}^{-1}$ i
s a minimal 
fractional decomposition of the fraction $AB^{-1}$, then there is a matrix 
differential operator $D$ over the linear closure of $\mc K$ such that $
A=A_{0}D$ and $B=B_{0}D$. 
Since $B$ is non-degenerate, $D={B_0}^{-1}B$ is actually a matrix 
differential operator over $\mc K$.
\end{proof}

\begin{remark}
We have the following two more equivalent definitions for a minimal fracitonal decomposition
$H=AB^{-1}$, with $A,B\in M_n(\mc K[\partial])$:
\begin{enumerate}
\item[$(iv)$]
the ``Bezout identity'' holds: $CA+DB=I$ for some $C,D\in M_n(\mc K[\partial])$,
\item[$(v)$]
$A$ and $B$ have kernels intersecting trivially over any differential field extension of $\mc K$.
\end{enumerate}
Condition $(v)$ obviously implies condition $(iii)$,
and, also, condition $(iv)$ implies condition $(v)$ since the identity matrix has zero kernel
over any field extension of $\mc K$.
To prove that $(iii)$ implies $(iv)$,
we use the fact that any left ideal of $M_n(\mc K[\partial])$ is principal (cf. \cite[Prop.4.10, p.82]{MR01}).
But if $E\in M_n(\mc K[\partial])$ is a generator of the left ideal generated by $A$ and $B$,
then by condition $(iii)$ we have that $\ker(\bar E)=0$,
and therefore $E$ must be invertible.
\end{remark}

\section{Maximal isotropicity of $\mc L_{A,B}$}
\label{sec:app4}

Let $A,B\in M_{n}({\mc K}[\partial])$ with $\det B\neq0$.
Recall that the subspace
$\mc L_{A,B}\subset\mc K^{n}\oplus\mc K^n$ (defined in the Introduction) 
is isotropic if and only if $AB^{-1}\in M_{n}({\mc K}(\partial))$
is skewadjoint, which in turn is equivalent to the following condition (\cite{DSK12},Proposition $6.5$):
\begin{equation}\label{20120119:eq1}
A^*B+B^*A=0\,.
\end{equation}
Hence, $\mc L_{A,B}\subset\mc K^n\oplus\mc K^n$
is maximal isotropic if and only if \eqref{20120119:eq1} and the following condition hold:
\begin{enumerate}
\item[$(vi)$]
if $G,H\in\mc K^n$ are such that $A^*H+B^*G=0$,
then there exists $F\in\mc K^n$ such that $G=AF$ and $H=BF$.
\end{enumerate}
\begin{theorem}\label{20120119:tim}
Suppose that $A,B\in M_n (\mc K[\partial])$ with $\det B\neq0$ 
satisfy equation \eqref{20120119:eq1}.
If $AB^{-1}$ is a minimal fractional decomposition,
then $\mc L_{A,B}\subset{\mc K}^{n}\oplus\mc K^n$ 
is a maximal isotropic subspace.
Namely, condition $(iii)$ of Section $6$ implies condition $(vi)$.
\end{theorem}
\begin{proof}
First, we prove the statement in the case when the differential field $\mc K$ is linearly closed.
Due to equation \eqref{20120119:eq1}, $A$ maps $\ker B$ to $\ker B^*$.
Since, by assumption, $\ker A\cap\ker B=0$,
this map is injective.
Moreover, since $\ker B$ and $\ker B^*$ have the same dimension
(equal to $d(B)$, by Lemma \ref{lem:3.1}(b)),
we conclude that we have a bijective map:
\begin{equation}\label{20120119:eq2}
A\,:\,\,\ker B\,\stackrel{\sim}{\longrightarrow}\ker B^*\,.
\end{equation}
Let $G,H\in\mc K^n$ be such that $A^*H+B^*G=0$.
Since $\det B\neq0$, we have that $B:\,\mc K^n\to\mc K^n$ is surjective
(by Lemma \ref{lem:3.1}(b)).
Hence we can choose $F_1\in\mc K^n$ such that $G=BF_1$.
Due to equation \eqref{20120119:eq1},
we get
$$
B^*AF_1=-A^*BF_1=-A^*G=B^*H\,.
$$
Hence, $H-AF_1\in\ker B^*$,
and by \eqref{20120119:eq2} there exists $F_2\in\ker B$ such that
$AF_2=H-AF_1$.
So, $H=A(F_1+F_2)$ and $G=BF_1=B(F_1+F_2)$,
proving condition $(vi)$.

Next, we prove the claim for a differential field $\mc K$ with algebraically closed subfield of constants.
Since, by assumption, $\ker\bar A\cap\ker\bar B=0$,
we know by the previous result that there is a solution $F\in {\mc L}^n$  
to the equations $G=AF$ and $H=BF$, where $\mc L$ is the linear closure 
of $\mc K$, 
and this solution is obviously unique
(since two solutions differ by an element in $\ker\bar A\cap\ker\bar B$).
We will next use a standard differential Galois theory argument to conclude 
that this solution $F$ must lie in $\mc K^n$.

By definition of the linear closure, all the entries of $F$ lie in some iterate Picard-Vessiot compositum of $\mc K$. Take $i$ minimal such that ${\mc K}_i$ satisfies this property. Assume $i \neq 0$. By Lemma {lem:5.9}, all entries of 
$F$ lie in some Picard-Vessiot subextension ${\mc K}_{i-1} \subset \mc P \subset {\mc K}_i$. As the solution $F$ is unique in the linear closure, it is fixed by all the differential automorphisms of the extension ${\mc K}_{i-1} \subset \mc P$, hence it lies in ${\mc K}_{i-1}^n$, which contradicts the minimality 
of $i$.
In the general case, we know from the previous discussion that there is a 
unique solution $F$ in $({{\mc K} \otimes_{\mc C} {\bar{\mc C}}})^n$. Hence 
all the entries of $F$ lie in a Galois extension $\mc G$ of $\mc K$. We know 
that there is a unique way to extend a derivation to an algebraic extension, 
so all algebraic automorphisms of this Galois extension are also differential 
automorphisms. Hence $F$ is fixed under the action of $Gal({\mc G}/{\mc K})$ 
which means that it lies in ${\mc K}^n$.

\end{proof}
\begin{proposition}
If a fraction $AB^{-1}$ of matrix differential operators satisfies 
$A^*B+B^*A=0$, and $AB^{-1}=({A_0D})({B_0D})^{-1}$ with
$Ker \bar{A_0} \cap Ker \bar{B_0}=0$, then ${\mc L}_{A,B}$ is maximal 
isotropic if and only if $D$ is surjective and 
$Ker D^* \cap (Im {A_0}^* + Im {B_0}^*)=0$.
\end{proposition}
\begin{proof}
Assume that $D$ is surjective and $Ker D^* \cap (Im {A_0}^* + Im {B_0}^*)=0$.
We have ${A_0}^*{B_0}+{B_0}^*{A_0}=0$, hence 
${\mc L}_{A_0,B_0}$ is maximal isotropic, since $A_0B_0^{-1}$ is a minimal 
fractional decomposition. 
Let $f,g \in {\mc K}^n$ be such that $A^*f+B^*g=0$. Since 
$Ker D^* \cap (Im {A_0}^* + Im {B_0}^*)=0$, we get that 
${A_0}^*f+{B_0}^*g=0$. By maximal isotropicity of ${\mc L}_{A_0,B_0}$, 
we can find some $h \in {\mc K}^n$ such that $f={B_0}h$ and $g={A_0}h$. Since 
$D$ is surjective, there is $k \in {\mc K}^n$ such that $h=Dk$. 
So $f=Ak$ and $g=Bk$, hence ${\mc L}_{A,B}$ is maximal isotropic.

Conversely, assume that ${\mc L}_{A,B}$ is maximal isotropic. First, we prove 
that $D$ is surjective. Take $f\in {\mc K}^n$. Multiplying on the left by 
$D^*$ the equation ${A_0}^*{B_0}+{B_0}^*{A_0}=0$ and evaluating it at $f$, 
we get that $A^*{B_0}f+B^*{A_0}f=0$, hence by maximal isotropicity of 
${\mc L}_{A,B}$, ${B_0}f=Bg$ and ${A_0}f=Ag$ for some $g \in {\mc K}^n$. 
Therefore $f-Dg \in Ker {A_0} \cap Ker {B_0}=0$, hence $f=Dg$. 
So $D$ is surjective. Next, take 
$x \in Ker D^* \cap (Im {A_0}^* + Im {B_0}^*)$. In particular, 
$x={A_0}^*g+{B_0}^*h$ for some $g, h \in {\mc K}^n$ and $A^*g+B^*h=0$. 
By maximal isotropicity of 
${\mc L}_{A,B}$, we 
see that $g=Bk$ and $h=Ak$ for some $k \in {\mc K}^n$. 
Multiplying the equation $A_0^*B_0+B_0^*A_0 =0$ by $Dk$ on the right, we get 
that $x=0$. 
\end{proof}
\begin{remark}
In the linearly closed case, a skewadjoint fraction $AB^{-1}$ is a minimal fractional decomposition if and only if $\mc L_{A,B}$  is maximal isotropic. 
Indeed, since $Ker D^* \cap (Im A_0^* + Im B_0^*)=0$ and 
$det(B_0^*)=\pm det(B_0) \neq 0$, we see that 
$B_0^*$ is surjective, hence $Ker D^*=0$. Therefore 
$d(D^*)=0=d(D)$ and $D$ is invertible.
Here we used Corollary \ref{cor:3.2} and Remark \ref{rem:3.3}.
\end{remark}%



\end{document}